\documentclass[12pt]{amsart}

\usepackage{amsthm,amsmath,amssymb}

\newtheorem{theorem}{Theorem}[section]

\newtheorem{lemma}[theorem]{Lemma}
\newtheorem{proposition}[theorem]{Proposition}

\theoremstyle{remark}

\newtheorem{question}[theorem]{Question}

\newcommand{\C}{\mathbb C}

\newcommand{\bC}{\mathbb C}

\newcommand{\bP}{\mathbb P}

\newcommand \re{\text{Re}\,}
\newcommand \im{\text{Im}\,}

\DeclareMathOperator{\esstype}{ess\, type}
\DeclareMathOperator{\mult}{mult}

\author{Peter Ebenfelt}
\address{Department of Mathematics, University of California at San Diego, La Jolla, CA 92093-0112}
\email{pebenfel@math.ucsd.edu}
\author{Duong Son}
\address{Department of Mathematics, University of California at San Diego, La Jolla, CA 92093-0112}
\email{snduong@math.ucsd.edu}
\thanks{The authors were supported in part by
DMS-1001322. The second author also acknowledges a scholarship from the Vietnam
Education Foundation.}


\begin{document}

\begin{abstract} We show that there are strictly pseudoconvex, real algebraic hypersurfaces in $\bC^{n+1}$ that cannot be locally embedded into a sphere in $\bC^{N+1}$ for any~$N$. In fact, we show that there are strictly pseudoconvex, real algebraic hypersurfaces in $\bC^{n+1}$ that cannot be locally embedded into any compact, strictly pseudoconvex, real algebraic hypersurface.
\end{abstract}
\title[On the existence of holomorphic embeddings into spheres]
{On the existence of holomorphic embeddings of strictly pseudoconvex algebraic hypersurfaces into spheres}
\maketitle

\section{Introduction}

In 1978, Webster \cite{Webster78b} proved that any real algebraic, Levi nondegenerate hypersurface $M\subset \bC^{n+1}$ admits a holomorphic mapping sending $M$ into a nondegenerate hyperquadric $Q\subset \bC^{N+1}$, for some $N\geq n$, such that the mapping is transversal to $Q$ and a local embedding near every point on $M$. In Webster's construction, however, the target hyperquadric $Q$ has in general different Levi signature from that of $M$. A few years later, Forstneric \cite{Forstneric86} and, independently, Faran \cite{Faran88} proved that there are strictly pseudoconvex (Levi signature 0), real {\it analytic} hypersurfaces in $\bC^{n+1}$ (indeed, in some sense a dense set of such for every $n$) that do not admit local holomorphic embeddings into a sphere $S\subset \bC^{N+1}$ (i.e.\ a hyperquadric with the same Levi signature) for {\it any} $N$. This leads naturally to the question (posed explicitly in \cite{Forstneric86}) whether there are strictly pseudoconvex, real {\it algebraic} hypersurfaces in $\bC^{n+1}$ that also do not admit local holomorphic embeddings into spheres of any dimension or if every such hypersurface can indeed be locally embedded into a sphere in a sufficiently high dimensional space; the proofs in \cite{Forstneric86} and \cite{Faran88} are highly nonconstructive (using the Baire Category Theorem) and do not offer any criteria for which hypersurfaces are embeddable. In the recent paper \cite{HZ09}, Huang and Zhang show that there are Levi nondegenerate, real algebraic hypersurfaces of Levi signature $l$ that cannot be locally embedded into a hyperquadric of the same signature, provided that the signature $l$ is strictly positive. The main idea in \cite{HZ09} is to establish a monotonicity property for the pseudoconformal curvature along null directions for the Levi form, an approach that cannot work when the hypersurfaces are strictly pseudoconvex. In this paper, we shall settle the strictly pseudoconvex case and show that there are strictly pseudoconvex, real algebraic hypersurfaces in $\bC^{n+1}$ that do not admit local holomorphic embeddings into a sphere $S\subset \bC^{N+1}$ for any $N$. In fact, we shall prove the stronger statement that there are strictly pseudoconvex, real algebraic hypersurfaces that do not admit nonconstant holomorphic mappings into any strictly pseudoconvex, {\it compact}, real algebraic hypersurface. We should point out that any local holomorphic mapping sending a real algebraic hypersurface into a strictly pseudoconvex, real algebraic hypersurface is necessarily algebraic \cite{Z99} and thus extends, in a sense, as a ``global'' mapping, albeit potentially multi-valued with singularities.

Other recent papers that have investigated the existence of holomorphic embeddings into hyperquadrics include \cite{KimOh09}, \cite{Oh07}, and \cite{Zaitsev08}. Zaitsev (in \cite{Zaitsev08}) studied obstructions to embeddability into hyperquadrics of any signature and dimension, and provided explicit examples of real analytic, Levi nondegenerate hypersurfaces that are not embeddable into hyperquadric of any dimension. In view of Webster's result mentioned above, the hypersurfaces in Zaitsev's examples are not real-algebraic. In fact, one can not make those hypersurfaces real algebraic by any changes of holomorphic coordinates (the first example of such a hypersurface was constructed by the first author in \cite{E96}; Forstneric \cite{Forstneric04} later showed that this is true for ``most" real analytic hypersurfaces). We also mention here the work of S.-Y. Kim and J.-W.~Oh concerning the existence of local embeddings into spheres using differential systems \cite{KimOh09}. The necessary and sufficient conditions for embeddability obtained in \cite{KimOh09} are very complicated relations between various pseudoconformal invariants, making them practically impossible to verify in specific examples, in particular in the examples given in this paper. The corresponding ``uniqueness'' or classification problem for mappings into hyperquadrics has also been extensively studied. We mention here only a few papers in this direction: \cite{Webster79,Faran82,Faran86,HuangJi01,EHZ04,EHZ05,BEH08,BEH09,ESh10,JPDLebl11} and refer the reader to these papers for further discussion and other references.
Finally we would like to mention that the existence of embeddings into spheres can be formulated as a ``sums of squares'' or positivity problem. The reader is referred to D'Angelo's paper \cite{JPD11} for details and discussion.

Recall that a smooth (meaning $C^\infty$ here) hypersurface $M\subset \bC^{n+1}$ is said to be {\it real algebraic} if it is contained in the zero locus of a non-trivial, real-valued polynomial. For the remainder of this paper, a real hypersurface refers to a connected, smooth manifold of real codimension one in a complex space (or manifold). Our first main result states that there are closed, strictly pseudoconvex, real algebraic hypersurfaces in $\bC^{n+1}$ that do not admit nonconstant holomorphic mappings into a compact, strictly pseudoconvex, real algebraic hypersurface (in particular, a sphere) in $\bC^{N+1}$ for any $N$. More precisely, we have:

\begin{theorem}\label{Main0} For any positive integer $n$, there exist closed, strictly pseudoconvex, real algebraic hypersurfaces $M$ in $\bC^{n+1}$ such that if $p\in M$ and $H$ is a local holomorphic mapping in an open neighborhood of $p$ sending $M$ into a compact, strictly pseudoconvex, real algebraic hypersurface in $\bC^{N+1}$, for some $N\geq n$, then $H$ is constant. \end{theorem}

As mentioned above, this result resolves, in particular, the question discussed in the first paragraph, and posed explicitly in \cite{Forstneric86} and \cite{HZ09} (Question 3.8), regarding the existence of local embeddings of strictly pseudoconvex, real algebraic hypersurfaces into spheres.

It is obvious that the source hypersurfaces whose existence is asserted in Theorem~\ref{Main0} cannot be compact (since the identity mapping would then be a counterexample). However, our second main result states that there are compact, real algebraic hypersurfaces in $\bC^{n+1}$ that are strictly pseudoconvex {\it except at one point} and that do not admit nonconstant holomorphic mappings into a compact, strictly pseudoconvex, real algebraic hypersurface (in particular, a sphere) in $\bC^{N+1}$ for any $N$:

\begin{theorem}\label{Main1} For any positive integer $n$, there exist compact, real algebraic hypersurfaces $M$ in $\bC^{n+1}$ that are strictly pseudoconvex except at one point such that if $p\in M$ and $H$ is a local holomorphic mapping in an open neighborhood of $p$ sending $M$ into a compact, strictly pseudoconvex, real algebraic hypersurface in $\bC^{N+1}$, for some $N\geq n$, then $H$ is constant. \end{theorem}

We should point out that the existence of weakly pseudoconvex points alone is not an obstruction to having holomorphic embeddings into spheres as examples readily show. For instance, the complex ellipsoid in $\bC^2$ defined by
\begin{equation}
|z|^{2p}+|w|^{2q}=1,\quad p,q\in \mathbb Z_+,\ p,q\geq 2
\end{equation}
is weakly pseudoconvex along $z=0$ and $w=0$, but is nevertheless mapped into the unit sphere in $\bC^2$ by the mapping $H(z,w)=(z^p,w^q)$, which is locally biholomorphic near every strictly pseudoconvex point.
Nevertheless, the idea in the proof of Theorem \ref{Main1} is to construct hypersurfaces with weakly pseudoconvex points of a particular type (see Section \ref{Example}) and then show that the existence of such points serve as global obstructions to embeddability into compact strictly pseudoconvex, real algebraic hypersurfaces (see Theorem \ref{cor}).  The proof of Theorem \ref{Main0} is then obtained by applying an automorphism of the projective space $\bP^{n+1}$ to move the weakly pseudoconvex point off to infinity.

In order to state our result concerning the obstruction to embeddability, we need to introduce a couple of basic concepts in CR geometry, namely those of finite nondegeneracy and minimality. The reader is referred to \cite{BER99a} for further elaboration on these and other concepts in CR geometry. Recall that a real hypersurface $M\subset \bC^{n+1}$ is said to be {\it minimal} at a point $p\in M$ if $M$ does not contain a complex hypersurface through $p$. If $M$ is real-analytic, then minimality is equivalent to  finite (commutator) type in the sense of Kohn and Bloom--Graham (see \cite{BER99a} for the precise definition of the latter notion). The hypersurface $M$ is said to be {\it finitely nondegenerate} at $p\in M$ if the following (infinite) collection of vectors spans $\bC^{n+1}$:
\begin{equation}\label{fnondeg}
(L^I\rho_Z)(p,\bar p),\quad I=(I_1,\ldots,I_n)\in \mathbb Z_+,
\end{equation}
where $L_1,\ldots,L_n$ form a basis for the CR vector fields on $M$ near $p$, we use multi-index notation $L^I:=L_1^{I_1}\ldots,L_n^{I_n}$, $\rho(Z,\bar Z)$ is a defining function for $M$ near $p$, and
\begin{equation}
\rho_Z:=\left(\frac{\partial\rho}{\partial Z_1},\ldots,\frac{\partial\rho}{\partial Z_{n+1}}\right).
\end{equation}
If $M$ is finitely nondegenerate at $p\in M$, then a finite subcollection of the vectors in \eqref{fnondeg} spans $\bC^{n+1}$. We say that $M$ is {\it $k$-nondegenerate} at $p$ if $k$ is the smallest integer such that the subcollection of vectors in \eqref{fnondeg} with $|I|:=I_1+\ldots+I_n\leq k$ spans. It is not difficult to check that $M$ is Levi nondegenerate at $p$ if and only if $M$ is $1$-nondegenerate at $p$. The notion of finite nondegeneracy and the integer $k$ are biholomorphic invariants and do not depend on the choice of $L_1,\ldots, L_n$ or $\rho$ (see \cite{BER99a}). We may now state our final main result in this paper, asserting that a minimal, real analytic hypersurface $M$ that contains a point $p\in M$ at which $M$ is $k$-nondegenerate with $k\geq 2$ (in particular, $M$ is not Levi nondegenerate at $p$) can only be mapped into a compact, strictly pseudoconvex, real algebraic hypersurface in the trivial way.

\begin{theorem}\label{cor}
Let $M$ be a real analytic hypersurface in $\bC^{n+1}$ which is minimal
at all points and assume that there is a point $p\in M$ at which $M$ is $k$-nondegenerate with $k\geq 2$. Suppose $q$ is an arbitrary point on $M$. If a local holomorphic mapping $H$ in an open neighborhood of~$q$ sends $M$ into a compact, strictly pseudoconvex, real algebraic hypersurface in $\bC^{N+1}$, for some $N\geq n$, then $H$ is constant.
\end{theorem}

The proof of this result has two main ingredients. The first is a result by Shafikov and Verma
\cite{ShafikovVerma07} concerning holomorphic continuation of a local holomorphic mapping into a compact, strictly pseudoconvex, real algebraic hypersurface along CR curves. The second is an analysis,  in the spirit of \cite{BR88} and \cite{ER06}, of the geometry of a holomorphic mapping near a finitely nondegenerate point.

\section{Preliminaries}
In this section, we recall some basic notions and fix some notation.
We refer to \cite{BER99a} for more detail and proofs of the statements made below.
Let $M\subset \bC^{n+1}$ be a smooth real-analytic hypersurface and $p\in M$.
There is a local holomorphic coordinate system $(z,w)$ vanishing at $p$
such that $M$ is given by
\[
\im w = \langle z,\bar z\rangle + F(z,\bar z, \re w),
\]
where $F$ vanishes to order $3$. The hypersurface $M$ is said to be {\it Levi-nondegenerate} at $p$ if the
quadratic form $\langle z,\bar z \rangle $ is nondegenerate and {\it strictly pseudoconvex} if the form is
positive definite. The coordinate system $(z,w)$ can be chosen such that $F(z,0,s)=F(0,\chi,s)\equiv 0$ (normal coordinates). By replacing $\im w = \frac1{2i}(w-\bar w)$ and $\re w = \frac12(w+\bar w)$
and solving for $w$ we obtain a defining equation for $M$ (in normal coordinates)
of the form:
\begin{equation}\label{complexdef}
w = Q(z,\bar z, \bar w),
\end{equation}
where $Q$ satisfies $Q(0,\bar z, \bar w) \equiv Q(z,0,\bar w) \equiv \bar w$ and
the reality condition
\begin{equation}
Q(z,\bar z, \bar Q(\bar z, z,w )) \equiv w.
\end{equation}
It is convenient to use the complex defining equation \eqref{complexdef}
to define the notions of essential finiteness and essential type
as follows. We replace $\bar z, \bar w$ by independent variables $\chi, \tau$
and write
\[
Q(z,\chi, 0) = \sum_{I\in \mathbb{N}^n} q_{I}(z)\chi^I.
\]
Let $\mathcal{I}_M$ be the ideal in $\bC[[z]]$ generated by $\{q_I(z)\}_{I}$.
Following Baouendi, Jacobowitz and Treves \cite{BJT85}, we shall say that $M$ is {\it essentially finite} at $p$ if $\mathcal{I}_M$ is of finite codimension
in $\bC[[z]]$. The dimension $\dim_{\bC} \bC[[z]]/{\mathcal{I}_M}$ is a biholomorphic
invariant of $M$ and is called the {\it essential type} of $M$ at $p$, denoted
by $\esstype_pM$.

Finite nondegeneracy, as defined in the introduction, at $(0,0)$ in normal coordinates reduces to the statement
that the following collection of vectors spans $\bC^n$:
\begin{equation}\label{fnondeg2}
Q_{z\chi^I}(0,0,0),\quad I=(I_1,\ldots, I_n)\in \mathbb Z_+,
\end{equation}
and $M$ is $k$-nondegenerate at $p=(0,0)$ if $k$ is the smallest integer such that the collection of vectors in \eqref{fnondeg2} with $|I|\leq k$ spans. To conclude this section, we mention that finite nondegeneracy of a real-analytic hypersurface
$M$ at $p\in M$ is equivalent to the statement that $M$ is essentially finite at $p$
and $\esstype_pM =1$ (see \cite[Proposition~11.8.27]{BER99a}).

\section{Some Lemmas and proof of Theorem \ref{cor}}
We will need several lemmas to prove Theorem \ref{cor}. First, we shall prove
the following Hopf Lemma type result:

\begin{lemma}\label{complexhopflemma}
Let $M, M'$ be real hypersurfaces through $p, p'$ in $\bC^{n+1}$ and
$\bC^{N+1}$ respectively. Assume that $M$ is minimal at $p$ and $M'$
has a smooth plurisubharmonic defining function near $p'$ (e.g.\ $M'$ strictly pseudoconvex at $p'$). Suppose that a germ of holomorphic mapping
$H:(\bC^{n+1},p) \to (\bC^{N+1},p')$ sends $M$ into $M'$. If $H$ does
not send an open neighborhood of $p$ in $\bC^{n+1}$ into $M'$ then $H$ is transversal to $M'$ at $p'=H(p)$.
\end{lemma}

Although this result is well known, the authors have been unable to find this precise statement in the literature. Thus, for the reader's convenience, we will provide a proof here. Let $\Delta$ be the unit disc in $\bC$ and $\overline \Delta$ its closure.
An analytic disc in $\bC^{n+1}$ is a continuous mapping $\sigma:\overline \Delta\to \bC^{n+1}$
which is holomorphic in $\Delta$. We shall say that $\sigma$ is attached to $M$ if
$\sigma(\partial \Delta) \subset M$. A fundamental theorem by Tumanov \cite{Tumanov88} (see also Chapter VIII in \cite{BER99a}) states that if $M$ is minimal at $p\in M$, then for any open neighborhood $U$ of $p$ the set of analytic discs of class $C^{1,\alpha}(\overline\Delta)$, for some $\alpha \in (0,1)$, contained in $U$,
attached to $M$, and passing through $p$ fill up an open subset of $U$ (containing at least ``one side" of $M$ near~$p$).

\begin{proof}[Proof of Lemma $\ref{complexhopflemma}$]
Since $H$ does not send a neighborhood of $p$ in $\bC^{n+1}$ into $M'$ and
$M$ is minimal, by the result of Tumanov mentioned above, there is an analytic
disc $\sigma: \overline \Delta \to \bC^{n+1}$ of class $C^{1,\alpha}(\overline\Delta)$, contained in an open neighborhood of $p$ on which $H$ is holomorphic and attached to $M$ with $\sigma(1) = p$, such that $H\left(\sigma (\overline{\Delta})\right)$ is not contained in $M'$. Now let $\rho'$ be a plurisubharmonic defining function for $M'$ near $p'$.
Then $u=\rho' \circ H \circ \sigma$ is a non-constant function of class $C^{1,\alpha}(\overline\Delta)$,
subharmonic in $\Delta$ and vanishing on $\partial \Delta$.
Let
$$X_p = \sigma_*\left(\frac{\partial}{\partial x}\biggl|_{\zeta = 1}\right), \quad \zeta = x +iy.$$
By the maximum principle, $u<0$ in $\Delta$ and by the classical Hopf boundary point lemma for subharmonic functions,
$$\frac{\partial u}{\partial x}(1) > 0$$ and consequently, $H_*X_p \not\in T_{p'} M'$.
Thus $H$ is transversal to $M'$ at $p'$.
\end{proof}

Recall that a germ of a holomorphic mapping $H\colon (\bC^{n+1},p)\to (\bC^{N+1},p')$ is said to be {\it finite} if the ideal $\mathcal{I}(H)$ generated by the components of $H$ in the ring $\mathcal O_p$ of germs of holomorphic functions at $p$ is of finite codimension. In this case, we shall refer to this codimension as the {\it multiplicity} of $H$ at $p$,
$$
\mult_p H:=\dim_{\bC} \mathcal O_p/\mathcal{I}(H).
$$
It is well known (see e.g.\ \cite{AGV85}) that if $H$ is finite at $p$, then for every $q$ close $p$ the number of preimages $m:=H^{-1}(H(q))$ is finite and $m\leq \mult_p H$. (In the equidimensional case $N=n$, the generic number of preimages equals $\mult_p H$, but in general we only have the inequality). If $H$ is finite at $p$ with
multiplicity one, then it follows readily that $H$ is a local embedding.

Now we can state and prove the following theorem about CR transversal mappings between
real hypersurfaces.
\begin{proposition}\label{multid}
Let $M$ and $M'$ be real-analytic hypersurfaces in
$\bC^{n+1}$ and $\bC^{N+1}$ respectively and let $p\in M$, $p'\in M'$. Suppose that
$H:(\bC^{n+1},p) \to (\bC^{N+1},p')$ is a germ of holomorphic mapping sending
$(M,p)$ into $M'$. If $M$ is essentially finite at $p$ and $H$ is transversal
to $M'$ at $p'=H(p)$, then $H$ is finite and
\begin{equation}\label{esstype:relation}
\mult_p H \leq \esstype_p M.
\end{equation}
Furthermore, if $M$ is finitely nondegenerate, then $H$ is a local embedding.
\end{proposition}
\begin{proof}
Suppose that $M$ and $M'$ are given in normal coordinates
$Z=(z,w)$ and $Z'=(z',w')$, vanishing at $p$ and $p'$ respectively, by complex defining functions $\rho$ and $\rho'$ of the forms:
\[
\rho(z,w,\bar z,\bar w) = w-Q(z,\bar z,\bar w), \quad \rho'(z',w',\bar z',\bar w') = w'-Q'(z',\bar z', \bar w').
\]
Since $H$ sends $M$ into $M'$, the following
holds for some real-analytic function $a(Z,\xi)$.
\begin{equation}\label{HsendsMintoMprime}
  G(Z) - Q'(F(Z), \bar H(\xi)) = a(Z,\xi) \,(w - Q(z,\xi)).
\end{equation}
Here, $H=(F,G)$ with $F=(F_1,\ldots, F_N)$. By setting $\xi = 0$, taking into account that $\bar H(0) = 0$, $Q(z,0,0) \equiv 0$ and
$Q'(z',0,0) \equiv 0$ we deduce that
\begin{equation}\label{Gw}
G(Z) = a(Z,0)\, w.
\end{equation}
Setting $w = \tau = 0$ and observing from \eqref{Gw} that $G(z,0) \equiv 0$ and $\bar G(\chi,0) \equiv 0$, we get
\begin{equation}\label{QQprime:rel}
 Q'(F(z,0), \bar F(\chi,0), 0) = a(z,0,\chi,0) \cdot Q(z,\chi,0).
\end{equation}
Since $H$ is transversal, we have $a(0) \ne 0$ (see e.g.\ \cite {BER07}). Therefore, $a(z,0,\chi,0)$ is non-vanishing for $(z,\chi)$ close to zero and hence
\begin{equation}\label{QQprime:rel2}
 a(z,0,\chi,0) ^{-1}\cdot Q'(F(z,0), \bar F(\chi,0), 0) = Q(z,\chi,0).
\end{equation}
We expand
\begin{equation}\label{expandQ}
Q(z,\chi,0) = \sum_I q_I(z)\,\chi^I.
\end{equation}
Let $\mathcal{I}_M$ and $\mathcal{I}(F)$ be the ideals in $\bC[[z]]$ generated by $\{q_I(z)\colon I\in \mathbb N^n\}$ and $\{F_j(z,0)\colon j=1,\dots N\}$, respectively. We claim that
\begin{equation}\label{inclusionideals}
\mathcal{I}_M \subset \mathcal{I}(F).
\end{equation}
Indeed, for each multi-index $I\in \mathbb{N}^n$, one has from \eqref{QQprime:rel2} that
\begin{equation}\label{diff}
q_I(z) = \frac{1}{I!}\,\frac{\partial^{I}}{\partial \chi^{I}}\biggl(a(z,0,\chi,0) ^{-1}\cdot Q'(F(z,0), \bar F(\chi,0), 0) \biggr) \biggr|_{\chi = 0}.
\end{equation}
If we expand
\begin{equation}\label{expandQ}
Q'(z',\chi',0) = \sum_J q'_J(z)\,(\chi')^J,
\end{equation}
then it is clear from \eqref{diff} that $q_I(z)$ belongs to the ideal generated by the $q'_J(F(z,0))$, $J\in\mathbb N^N$, which in turn belongs to the ideal $\mathcal {I}(F)$ (since the ideal $\mathcal{I}_{M'}$, generated by the $q'_J(z')$, of course is contained in the maximal ideal).  Therefore, we obtain \eqref{inclusionideals}.
Furthermore, since $M$ is essentially finite, $\mathcal{I}_M$ is of finite codimension in $\bC[[z]]$ and so is $\mathcal{I}(F)$, by \eqref{inclusionideals}, and hence $F(z,0)$ is finite. Moreover,
\begin{equation}\label{e:16}
\mult_0(F(\cdot,0)) = \dim_{\bC}\bC[[z]]/\mathcal{I}(F) \leq \dim_{\bC}\bC[[z]]/\mathcal{I}_M = \esstype_0(M).
\end{equation}
On the other hand, it follows from \eqref{Gw} and the invertibility of $a(Z,0)$ that $w\in \mathcal{I}(H)$ and, hence, $H$ is also finite and
\begin{equation}\label{e:17}
\mult_0(H) = \mult_0(F(\cdot,0)).
\end{equation}
From \eqref{e:16} and \eqref{e:17}, we obtain \eqref{esstype:relation}.

Furthermore, if $M$ is finitely nondegenerate at $p$, or equivalently (as mentioned above) $\esstype_p M=1$, then it follows from \eqref{esstype:relation} that $\mult_p H =1$ and thus $H$ is a local embedding.
\end{proof}

The following lemma is elementary.
\begin{lemma}\label{leviidentity}
Let $M$ and $M'$ be real hypersurfaces in $\bC^{n+1}$ and $\bC^{N+1}$ through $p$ and $p'$
respectively. Let $H:(\bC^n,p) \to (\bC^N,p')$ be a germ of holomorphic mapping
sending $M$ into $M'$. Suppose that $M'$ is strictly pseudoconvex at $p'$ and $H$
is local embedding. Then $M$ is strictly pseudoconvex at~$p$.
\end{lemma}
\begin{proof}
Let $\rho'$ be a strictly plurisubharmonic local defining function for $M'$ near $p$.
By Lemma~\ref{complexhopflemma}, $H$ is transversal to $M'$ at $p'$ and hence
$\rho=\rho'\circ H$ is a local defining function for $M$ near~$p$ (see \cite{BER07}). Since
$H$ is local embedding, $\rho$ is
strictly plurisubharmonic near~$p$ and thus $M$ is strictly pseudoconvex at $p$.
\end{proof}

In order to prove Theorem~\ref{cor}, we will also need the following theorem
by Shafikov and Verma.

\begin{theorem}[\cite{ShafikovVerma07}]\label{shafikov:verma}
Let $M$ be a smooth real analytic minimal hypersurface in $\C^{n+1}$, $M'$
a compact, strictly pseudoconvex, real algebraic hypersurface in
$\C^{N+1}$, $1<n\le N$. Suppose that $H$ is a germ of a holomorphic map at a
point $p\in M$ and $H(M)\subset M'$. Then $H$ extends as a holomorphic
map along any CR-curve on $M$ such that the extension sends $M$ into $M'$.
\end{theorem}

\begin{proof}[Proof of Theorem~\ref{cor}] Let $M'$ be a compact, strictly pseudoconvex, real algebraic hypersurface in $\bC^{N+1}$. Let $q\in M$, $q'\in M'$ and
suppose that $H:(\bC^{n+1}, q) \to (\bC^{N+1},q')$  is a germ of a nonconstant holomorphic mapping
sending $(M,q)$ into $(M',q')$. By Theorem~\ref{shafikov:verma}, $H$ extends as a holomorphic map (necessarily nonconstant)
along every CR-curve, in particular one connecting $q$ and $p$ (which exists by the assumption that $M$ is minimal). The extension is then a nonconstant holomorphic mapping near $p$ sending $M$ into $M'$. Since $M'$ is strictly pseudoconvex, $H$ does not send an open neighborhood of $p$ in $\bC^{n+1}$ into $M'$ and, therefore by Lemma~\ref{complexhopflemma}, $H$ is transversal to $M'$ at $p'$.
Consequently, by Proposition \ref{multid}, $H$ is a local embedding at $p$, since $M$ is finitely ($k$-)nondegenerate at $p$.
It then follows from Lemma~\ref{leviidentity} that $M$ is strictly pseudoconvex at $p$. This is a contradiction (since $k\geq 2$), which completes the proof of Theorem \ref{cor}.
\end{proof}

\section{Proof of Theorems $\ref{Main0}$ and $\ref{Main1}$; A Construction}\label{Example}

In this section we shall construct a compact, real algebraic hypersurface $M$ in $\bC^{n+1}$ that is strictly pseudoconvex except at $0\in M$ and that is $3$-nondegenerate at $0$. This will prove Theorem \ref{Main1} in view of Theorem \ref{cor}. (Finite nondegeneracy implies, in particular, minimality; see \cite{BER99a}.) At the end of this section, we shall also prove Theorem \ref{Main0}.

Consider the homogeneous, real-valued polynomials $P_R$ in $\bC^{n+1}$ (where we shall use the notation $Z=(z,w)\in \bC^n\times\bC$ for the variables) given by
\begin{equation}\label{PR}
P_R(z,w,\bar z,\bar w):=R\left(\sum_{k=1}^n|z_k|^2+|w|^2\right)^2+2\sum_{k=1}^n\re(z_k\bar z_k^3),\quad R>0.
\end{equation}
It is easy to see that for sufficiently large $R>0$, the polynomials $P_R$ are strictly positive and strictly plurisubharmonic away from $0$. In what follows, we shall assume that $R$ is so large that these facts hold and will drop the subscript $R$ on $P_R$, i.e.\ $P=P_R$. Next, consider the polynomial
\begin{equation}\label{rho}
\rho(z,w,\bar z,\bar w):=-\im w+P(z,w,\bar z,\bar w)
\end{equation}
and define $M\subset \bC^{n+1}$ to be real algebraic variety given by the zero locus of $\rho$. It is straightforward to verify that $M$ is a compact ($\rho\to \infty$ as $||z||^2+|w|^2\to \infty$), real hypersurface (no singularities; $\rho_w=0$ implies that $w$ has strictly negative imaginary part and there are no such points on $M$). Since $\partial\bar \partial \rho=\partial\bar\partial P$, we conclude that $\rho$ is strictly plurisubharmonic except at $0\in M$ and, hence, $M$ is strictly pseudoconvex except at $0$ (where clearly $M$ is weakly pseudoconvex). To show that $M$ satisfies the conclusion of Theorem \ref{Main1}, it suffices to show that $M$ is $k$-nondegenerate at $0$ with $k\geq 2$ (and in this case $k$ equals $3$). A basis for the CR vector fields on $M$ near $0$ is given by
\begin{equation}\label{CRvector}
L_j:=\frac{\partial}{\partial \bar z_j}-2iP_{\bar z_j}\frac{\partial}{\partial \bar w},\quad j=1,\ldots n.
\end{equation}
We note that for any multi-index $I\in \mathbb Z_+^n$ and any smooth function $u$, we have $$(L^Iu)(0,0)=\frac{\partial^{|I|} u}{\partial \bar z^I}(0,0).$$
We also have
\begin{multline}
\rho_Z=\bigg(2R\bar z_1\big(\sum_{k=1}^n|z_k|^2+|w|^2\big)+\bar z_1^3+3z_1^2\bar z_1,\ldots, \\2R\bar z_n\big(\sum_{k=1}^n|z_k|^2+|w|^2\big)+\bar z_n^3+3z_n^2\bar z_1, -\frac{1}{2i}+2R\bar w\big(\sum_{k=1}^n|z_k|^2+|w|^2\big)\bigg)
\end{multline}
Thus, we conclude that $\rho_Z(0,0)=(0,\ldots,0,-1/2i)$ and $(L^I\rho_Z)(0,0)=(0,\ldots,0)$ unless $L^I=L_j^3$ in which case we have $(L_j^3\rho_Z)(0,0)=(0,\ldots,0,6,0,\ldots,0)$ where the 6 appears in the $j$:th component. Therefore, $M$ is $3$-nondegenerate at $0$. As mentioned above, $M$ then satisfies the conclusion of Theorem \ref{Main1}, and this completes the proof of this theorem.

To prove Theorem \ref{Main0}, we note that $M$ meets the complex hyperplane $W:=\{(z,w)\in \bC^{n+1}\colon w=0\}$ only at the point $(0,0)$. Consider the linear fractional transformation
$$
\hat z_j=\frac{z_j}{w},\quad \hat w=\frac{1}{w},\quad j=1,\ldots n.
$$
This transformation extends as an automorphism $\Phi$ of the projective space $\bP^{n+1}$ sending the complex hyperplane $W$ to the hyperplane at infinity. If we define $\hat M:=\Phi(M)\cap \bC^{n+1}$, then $\hat M$ is a closed, real algebraic hypersurface in $\bC^{n+1}$ that is biholomorphic to $M\setminus\{(0,0)\}$ and is therefore strictly pseudoconvex at every point. Clearly, any local holomorphic mapping of $\hat M$ into a compact, strictly pseudoconvex, real algebraic hypersurface in $\bC^{N+1}$ induces a local holomorphic mapping of $M$ into the same hypersurface and, hence, must be constant by Theorem \ref{Main1}. This completes the proof of Theorem \ref{Main0}.

\section{Concluding Remarks}

While, as mentioned in the introduction, it is obvious that there cannot be a compact, real algebraic hypersurface $M\subset \bC^{n+1}$ that is strictly pseudoconvex at every point and satisfies the conclusion of Theorem \ref{Main1}, there could exist  such a hypersurface that cannot be locally (nontrivially) mapped into a sphere in $\bC^{N+1}$ for any $N$. Thus, to conclude this paper we offer the following refined, yet unresolved version of the original question settled in this paper:

\begin{question}
Does there exist a {\it compact}, real algebraic hypersurface $M\subset \bC^{n+1}$, strictly pseudoconvex at {\it every} point such that if $H$ is a local mapping sending a piece of $M$ into a sphere in $\bC^{N+1}$, then $H$ must be constant?
\end{question}

\def\cprime{$'$}


\begin{thebibliography}{29}

\bibitem{AGV85}
{\sc Arnol{\cprime}d, V.~I., Guse{\u\i}n-Zade, S.~M., and Varchenko, A.~N.}
\newblock {\em Singularities of differentiable maps. {V}ol. {I}}, vol.~82 of
  {\em Monographs in Mathematics}.
\newblock Birkh\"auser Boston Inc., Boston, MA, 1985.
\newblock The classification of critical points, caustics and wave fronts,
  Translated from the Russian by Ian Porteous and Mark Reynolds.

\bibitem{BEH08}
{\sc Baouendi, M.~S., Ebenfelt, P., and Huang, X.}
\newblock Super-rigidity for {CR} embeddings of real hypersurfaces into
  hyperquadrics.
\newblock {\em Adv. Math. 219}, 5 (2008), 1427--1445.

\bibitem{BEH09}
{\sc Baouendi, M.~S., Ebenfelt, P., and Huang, X.}
\newblock Holomorphic mappings between hyperquadrics with small signature
  difference.
\newblock {\em Amer. J. Math. 133}, 6 (2011), 1633--1661.

\bibitem{BER99a}
{\sc Baouendi, M.~S., Ebenfelt, P., and Rothschild, L.~P.}
\newblock {\em Real submanifolds in complex space and their mappings}, vol.~47
  of {\em Princeton Mathematical Series}.
\newblock Princeton University Press, Princeton, NJ, 1999.

\bibitem{BER07}
{\sc Baouendi, M.~S., Ebenfelt, P., and Rothschild, L.~P.}
\newblock Transversality of holomorphic mappings between real hypersurfaces in
  different dimensions.
\newblock {\em Comm. Anal. Geom. 15\/} (2007), 589--611.

\bibitem{BJT85}
{\sc Baouendi, M.~S., Jacobowitz, H., and Treves, F.}
\newblock On the analyticity of {CR} mappings.
\newblock {\em Ann. Math. 122\/} (1985), 365----400.

\bibitem{BR88}
{\sc Baouendi, M.~S., and Rothschild, L.~P.}
\newblock Germs of {CR} maps between real analytic hypersurfaces.
\newblock {\em Invent. Math. 93}, 3 (1988), 481--500.

\bibitem{JPD11}
{\sc D'Angelo, J.~P.}
\newblock Hermitian analogues of {H}ilbert's 17-th problem.
\newblock {\em Adv. Math. 226}, 5 (2011), 4607--4637.

\bibitem{JPDLebl11}
{\sc D'Angelo, J.~P., and Lebl, J.}
\newblock Hermitian symmetric polynomials and {CR} complexity.
\newblock {\em J. Geom. Anal. 21}, 3 (2011), 599--619.

\bibitem{E96}
{\sc Ebenfelt, P.}
\newblock On the unique continuation problem for {CR} mappings into nonminimal
  hypersurfaces.
\newblock {\em J. Geom. Anal. 6}, 3 (1996), 385--405 (1997).

\bibitem{EHZ04}
{\sc Ebenfelt, P., Huang, X., and Zaitsev, D.}
\newblock Rigidity of {CR}-immersions into spheres.
\newblock {\em Comm. Anal. Geom. 12}, 3 (2004), 631--670.

\bibitem{EHZ05}
{\sc Ebenfelt, P., Huang, X., and Zaitsev, D.}
\newblock The equivalence problem and rigidity for hypersurfaces embedded into
  hyperquadrics.
\newblock {\em Amer. J. Math. 127}, 1 (2005), 169--191.

\bibitem{ER06}
{\sc Ebenfelt, P., and Rothschild, L.~P.}
\newblock Transversality of {CR} mappings.
\newblock {\em Amer. J. Math. 128\/} (2006), 1313--1343.

\bibitem{ESh10}
{\sc Ebenfelt, P., and Shroff, R.}
\newblock Partial rigidity of {CR} embeddings of real hypersurfaces into
  hyperquadrics with small signature difference.
\newblock {\em Comm. Anal. Geom.\/} (to appear).

\bibitem{Faran82}
{\sc Faran, J.~J.}
\newblock Maps from the two-ball to the three-ball.
\newblock {\em Invent. Math. 68}, 3 (1982), 441--475.

\bibitem{Faran86}
{\sc Faran, J.~J.}
\newblock The linearity of proper holomorphic maps between balls in the low
  codimension case.
\newblock {\em J. Differential Geom. 24}, 1 (1986), 15--17.

\bibitem{Faran88}
{\sc Faran, V, J.~J.}
\newblock The nonimbeddability of real hypersurfaces in spheres.
\newblock {\em Proc. Amer. Math. Soc. 103}, 3 (1988), 902--904.

\bibitem{Forstneric86}
{\sc Forstneri{\v{c}}, F.}
\newblock Embedding strictly pseudoconvex domains into balls.
\newblock {\em Trans. Amer. Math. Soc. 295}, 1 (1986), 347--368.

\bibitem{Forstneric04}
{\sc Forstneri{\v{c}}, F.}
\newblock Most real analytic {C}auchy-{R}iemann manifolds are nonalgebraizable.
\newblock {\em Manuscripta Math. 115}, 4 (2004), 489--494.

\bibitem{HuangJi01}
{\sc Huang, X., and Ji, S.}
\newblock Mapping $\mathbb{B}^n$ into $\mathbb{B}^{2n-1}$.
\newblock {\em Inventiones Mathematicae 145\/} (2001), 219--250.
\newblock 10.1007/s002220100140.

\bibitem{HZ09}
{\sc Huang, X., and Zhang, Y.}
\newblock Monotonicity for the {Chern-Moser-Weyl} curvature tensor and {CR}
  embeddings.
\newblock {\em Science in China, Ser. A 52}, 12 (2009), 2617--2627.

\bibitem{KimOh09}
{\sc Kim, S.-Y., and Oh, J.-W.}
\newblock Local embeddability of {CR} manifolds into spheres.
\newblock {\em Math. Ann. 344}, 1 (2009), 185--211.

\bibitem{Oh07}
{\sc Oh, J.-W.}
\newblock Local {C}auchy-{R}iemann embeddability of real hyperboloids into
  spheres.
\newblock {\em Proc. Amer. Math. Soc. 135}, 2 (2007), 397--403 (electronic).

\bibitem{ShafikovVerma07}
{\sc Shafikov, R., and Verma, K.}
\newblock Extension of holomorphic maps between real hypersurfaces of different
  dimension.
\newblock {\em Ann. Inst. Fourier (Grenoble) 57}, 6 (2007), 2063--2080.

\bibitem{Tumanov88}
{\sc Tumanov, A.~E.}
\newblock Extension of {CR}-functions into a wedge from a manifold of finite type.
  ({Russian}).
\newblock {\em Mat. Sb. (N.S) 136 (178)\/} (1988), 128--130.

\bibitem{Webster78b}
{\sc Webster, S.~M.}
\newblock Some birational invariants for algebraic real hypersurfaces.
\newblock {\em Duke Math. J. 45}, 1 (1978), 39--46.

\bibitem{Webster79}
{\sc Webster, S.~M.}
\newblock The rigidity of {C}-{R} hypersurfaces in a sphere.
\newblock {\em Indiana Univ. Math. J. 28}, 3 (1979), 405--416.

\bibitem{Z99}
{\sc Zaitsev, D.}
\newblock Algebraicity of local holomorphisms between real-algebraic
  submanifolds of complex spaces.
\newblock {\em Acta Math. 189}, 3 (1999), 273--305.

\bibitem{Zaitsev08}
{\sc Zaitsev, D.}
\newblock Obstructions to embeddability into hyperquadrics and explicit
  examples.
\newblock {\em Math. Ann. 342}, 3 (2008), 695--726.

\end{thebibliography}
\end{document}